\documentclass[12pt]{article}
\usepackage{amsmath,amssymb,amsbsy,amsfonts,amsthm,latexsym,amsopn,amstext,                                                amsxtra,euscript,amscd}
\begin{document}

\newtheorem{theorem}{Theorem}
\newtheorem{lemma}[theorem]{Lemma}
\newtheorem{claim}[theorem]{Claim}
\newtheorem{cor}[theorem]{Corollary}
\newtheorem{prop}[theorem]{Proposition}
\newtheorem{definition}{Definition}
\newtheorem{question}[theorem]{Open Question}
\newtheorem{conj}[theorem]{Conjecture}
\newtheorem{prob}{Problem}
\newtheorem{algorithm}[theorem]{Algorithm}

\def\squareforqed{\hbox{\rlap{$\sqcap$}$\sqcup$}}
\def\qed{\ifmmode\squareforqed\else{\unskip\nobreak\hfil
\penalty50\hskip1em\null\nobreak\hfil\squareforqed
\parfillskip=0pt\finalhyphendemerits=0\endgraf}\fi}

\def\cA{{\mathcal A}}
\def\cB{{\mathcal B}}
\def\cC{{\mathcal C}}
\def\cD{{\mathcal D}}
\def\cE{{\mathcal E}}
\def\cF{{\mathcal F}}
\def\cG{{\mathcal G}}
\def\cH{{\mathcal H}}
\def\cI{{\mathcal I}}
\def\cJ{{\mathcal J}}
\def\cK{{\mathcal K}}
\def\cL{{\mathcal L}}
\def\cM{{\mathcal M}}
\def\cN{{\mathcal N}}
\def\cO{{\mathcal O}}
\def\cP{{\mathcal P}}
\def\cQ{{\mathcal Q}}
\def\cR{{\mathcal R}}
\def\cS{{\mathcal S}}
\def\cT{{\mathcal T}}
\def\cU{{\mathcal U}}
\def\cV{{\mathcal V}}
\def\cW{{\mathcal W}}
\def\cX{{\mathcal X}}
\def\cY{{\mathcal Y}}
\def\cZ{{\mathcal Z}}

\def\fI{{\mathfrak I}}
\def\fJ{{\mathfrak J}}

\def\MNL{{\mathfrak M}(N;K,L)}
\def\VNL{V_m(N;K,L)}
\def\RNL{R(N;K,L)}

\def\MNm{{\mathfrak M}_m(N;K)}
\def\VNm{V_m(N;K)}

\def\Xm{\cX_m}

\def \C {{\mathbb C}}
\def \F {{\mathbb F}}
\def \L {{\mathbb L}}
\def \K {{\mathbb K}}
\def \Q {{\mathbb Q}}
\def \Z {{\mathbb Z}}

\def\barG{\overline{\cG}}
\def\\{\cr}
\def\({\left(}
\def\){\right)}
\def\fl#1{\left\lfloor#1\right\rfloor}
\def\rf#1{\left\lceil#1\right\rceil}

\newcommand{\pfrac}[2]{{\left(\frac{#1}{#2}\right)}}

\def\rem{\mathrm{\, rem~}}

\def \Prob{{\mathrm {}}}
\def\e{\mathbf{e}}
\def\ep{{\mathbf{\,e}}_p}
\def\epp{{\mathbf{\,e}}_{p^2}}
\def\em{{\mathbf{\,e}}_m}
\def\eps{\varepsilon}
\def\Res{\mathrm{Res}}
\def\vec#1{\mathbf{#1}}

\def \li {\mathrm {li}\,}

\def\mand{\qquad\mbox{and}\qquad}

\newcommand{\comm}[1]{\marginpar{%
\vskip-\baselineskip 
\raggedright\footnotesize
\itshape\hrule\smallskip#1\par\smallskip\hrule}}

\title{Pseudorandomness and Dynamics of Fermat Quotients}

\author{ 
{\sc Alina~Ostafe}\\
{Institut f\"ur Mathematik, Universit\"at Z\"urich}\\
{Winterthurerstrasse 190 CH-8057, Z\"urich, Switzerland}\\
{\tt alina.ostafe@math.uzh.ch}
\and
{\sc Igor E.~Shparlinski} \\
{Department of Computing, Macquarie University} \\
{Sydney, NSW 2109, Australia} \\
{\tt igor@comp.mq.edu.au}}

\date{}

\maketitle

\begin{abstract} We obtain some theoretic and experimental 
results concerning various  properties (the number of fixed points, 
image distribution, cycle lengths) of the
dynamical system naturally associated with Fermat quotients
acting on the set $\{0, \ldots, p-1\}$. We also consider 
pseudorandom properties of Fermat quotients such as
joint distribution and linear complexity. 
\end{abstract}



\section{Introduction}

\subsection{Background}

For a prime $p$ and an integer $u$ with $\gcd(u,p)=1$ 
the {\it Fermat quotient\/} $q_p(u)$ is defined as the unique integer
with 
$$
q_p(u) \equiv \frac{u^{p-1} -1}{p} \pmod p, \qquad 0 \le q_p(u) \le p-1, 
$$
and we also define 
$$
q_p(kp) = 0, \qquad k \in \Z.
$$

It is well-known that the $p$-divisibility of Fermat
quotients $q_p(a)$ by $p$ has numerous applications, which include
the Fermat Last Theorem and squarefreeness testing,
see~\cite{ErnMet2,Fouch,Gran1,Len}.
In particular, the smallest value $\ell_p$ of $u\ge 1$ for which
$q_p(u) \not \equiv 0 \pmod p$ plays a prominent role in
these applications, for which the 
following estimates are given~\cite{BFKS}
$$
\ell_p \le \left\{\begin{array}{lll}
 (\log p)^{463/252 + o(1)}  &\quad \text{for all}\ p, \\
 (\log p)^{5/3 + o(1)}  &\quad \text{for almost all}\ p, 
\end{array}\right.
$$
(where almost all $p$ means for all $p$ but a set of relative density zero), 
which improve the previous estimates of the form
$\ell_p  = O\( (\log p)^2\)$ of~\cite{Fouch,Gran2,Ihara,Len}.
It is widely believed that $\ell_p = 2$ for all primes $p$, 
except for a very thin set of so called {\it Wieferich primes\/},
which one expects $\ell_p =3$ (in particular, it is expected 
that $\ell_p \le 3$ for all primes).  The behaviour (and even the infinitude) 
of Wieferich primes is still very poorly understood, although 
several interesting results, relating Wieferich primes 
to other number theoretic problems are known, see~\cite{GrSo,MoMu,Silv}.

There are also several results about 
the distribution  of Fermat quotients. For instance, 
Heath-Brown~\cite{H-B} has proved that the Fermat quotients  $q_p(u) $ are asymptotically uniformly distributed (after 
scaling by $1/p$ and mapping them into $q_p(u)/p \in [0,1]$) 
for $u = M+1, \ldots, M+N$ for any integers $M$ 
and $N \ge p^{1/2 +\varepsilon}$ for some fixed $\varepsilon$ and
$p\to \infty$. Note that~\cite[Theorem~2]{H-B} gives this 
only for $N  \ge p^{3/4 +\varepsilon}$ but using the full
strength of the Burgess bound one can lower this
threshold down to $h \ge p^{1/2 +\varepsilon}$, see 
Lemma~\ref{lem:HB} below and 
also~\cite[Section~4]{ErnMet2}. 

It is also shown in~\cite[Proposition~2.1]{Fouch} 
that for any integer $a$ the number of 
solutions to the equation  $q_p(u) = a$, $0 \le u < p$, is at most
\begin{equation}
\label{eq:Rep}
\# \{u \in \{0, \ldots, p-1\}\ :\ q_p(u) = a\} \le p^{1/2 + o(1)}. 
\end{equation}

Finally, we also recall several results on congruences 
involving Fermat quotients, see~\cite{AgoSkul,Di-BFa,Sun} and references 
therein.

\subsection{Our results}

Here we consider the dynamical system generated by Fermat quotients.
That is, we fix a sufficiently large prime $p$ 
and, for an   initial value $u_0  \in \{0, \ldots, p-1\}$
we consider the sequence
\begin{equation}
\label{eq:FermDyn}
u_n = q_p(u_{n-1}), \qquad n =1, 2, \ldots\,.
\end{equation}

Clearly, there is some $t$ such that $u_t = u_k$
for some $k < t$. Then   $u_{n+t} = u_{n+k}$ for any $n \ge 0$.
Accordingly, for the smallest value of $t$ with the above
condition,  we call $u_0, \ldots, u_{t-1}$ the orbit 
of the initial value $u_0$. 

Here we address various questions concerning  the 
sequences generated by~\eqref{eq:FermDyn} such as 
the  number of fixed points, image size and 
the ``typical'' orbit length.
In particular, we compare their characteristics with
those expected from random maps, see~\cite{FlOdl}.
All our numerical results support the natural expectation that
the map  
$u\mapsto q_p(u)$ behaves very similar  to a random map on 
the set $\{0, \ldots, p-1\}$.

 We also investigate 
their distribution and other characteristics which are 
relevant to their use as pseudorandom number generators.
As we have mentioned, a result of Heath-Brown~\cite{H-B}
implies that the fractions $q_p(u)/p$ are uniformly 
distributed for $u = M+1, \ldots, M+N$,  
provided that  $N \ge p^{1/2 +\varepsilon}$ for some 
fixed $\varepsilon > 0$. However, the method of~\cite{H-B},
based on bounds of multiplicative character sums, 
such as    the
Polya-Vinogradov  and Burgess bounds,
see~\cite[Theorems~12.5 and~12.6]{IwKow},
does not seem to apply to studying the distribution of
several consecutive elements (as it is essentially equivalent 
to estimating short sums of multiplicative characters 
modulo $p^2$ with polynomial
arguments).
 Here we use a 
different approach, 
to study the distribution of points 
\begin{equation}
\label{eq:Points}
\(\frac{q_p(u)}{p},\ldots, \frac{q_p(u+s-1)}{p}\), \qquad u = M+1, \ldots, M+N, 
\end{equation}
in the $s$-dimensional cube, which is nontrivial 
provided that $N \ge p^{1 +\varepsilon}$ for any fixed real  $\varepsilon > 0$
and integer $s \ge 1$.

We also obtain a nontrivial lower bound on the linear 
complexity of the sequence $q_p(u)$ which is 
also a very important characteristic of any sequence
relevant to its applications to  both cryptography and 
Quasi-Monte Carlo methods,  see~\cite{CDR,MOV,TopWin}. 

Besides theoretic estimates, we also present results
of several numerical tests. Some of these tests 
are based 
on a  modification of an algorithm described 
in~\cite{ErnMet1,ErnMet2}, which seems to be more 
computationally efficient. We also address some
other algorithmic aspects of computation with 
Fermat quotients. In particular, we give asymptotic estimates 
of several new  algorithms which 
we design for this purpose.

We note that all heuristic predictions 
concerning various conjectures about Fermat quotinets
(for example, the expected number of Wieferich primes 
up to $x$ as $x \to \infty$) are based on the assumption
of the pseudorandomness of the map $u \mapsto q_p(u)$. Our results 
provide some theoretic and experimental support to this assumption
which seems to be never  systematically verified prior to our 
work. 

Finally, motivated by the pseudorandom nature of the map 
$u \mapsto q_p(u)$, we also discuss some possibilities of using 
Fermat quotients for designing cryptographically useful
hash functions.

We remark that Smart and Woodcock~\cite{SmWo} 
have considered iterations of a related function
\begin{equation}
\label{eq:L fun}
L_p(u) = \frac{u^{p} -u}{p} 
\end{equation}
in the ring of $p$-adic integers. However,  the setting
of~\cite{SmWo} (where $p$ is fixed, for example $p = 2$)
and our settings where $p$ is the main growing parameter are
very different.

\subsection{Acknowledgement}

The authors are very grateful to Sergei Konyagin for his comments which have led to a significant improvement of the preliminary 
version of Theorem~\ref{thm:Image Size}.
Thanks also go to Daniel Sutantyo for 
his help with Magma programs and Tauno Mets{\"a}nkyl{\"a} for 
his  comments and encouragement. 

During the preparation of this paper,  
A.~O. was supported in part by 
the Swiss National Science Foundation   Grant~121874  
and I.~S. by
the  Australian Research Council 
Grant~DP0556431.

\section{Preparations}

\subsection{General Notation}

Throughout the paper,  $p$  always denotes   prime 
numbers, while $k$, $m$ and $n$ (in both the upper and
lower cases) denote positive integer 
numbers. 

For  integers $a$, $b$ and  $m \ge 1$ with $\gcd(b,m)=1$, 
we write  
$$
c = a/b~\rem m
$$ 
for the unique integer $c$
with $bc \equiv a \pmod m$ and $0 \le c < m$.

We also define
$$
\ep(z) = \exp(2 \pi i z/p).
$$

The implied constants in the symbols `$O$',
and `$\ll$'  
may occasionally depend on an integer parameter $s$
and are absolute otherwise.  
We recall that the notations $U = O(V)$ and $V \ll U$ are both
equivalent to the assertion that the inequality $|U|\le cV$ holds for some
constant $c>0$.

\subsection{Discrepancy and linear complexity} 

Given a sequence $\Gamma$ of $N$ points 
\begin{equation}
\label{eq:GenSequence}
\Gamma = \left\{(\gamma_{n,1}, \ldots, \gamma_{n,s})_{n=0}^{N-1}\right\}
\end{equation}
in the $s$-dimensional unit cube $[0,1)^s$
it is natural to measure the level of its statistical uniformity 
in terms of the {\it discrepancy\/} $\Delta(\Gamma)$. 
More precisely, 
$$
\Delta(\Gamma) = \sup_{B \subseteq [0,1)^s}
\left|\frac{T_\Gamma(B)} {N} - |B|\right|,
$$
where $T_\Gamma(B)$ is the number of points of  $\Gamma$
inside the box
$$
B = [\alpha_1, \beta_1) \times \ldots \times [\alpha_{s}, \beta_{s})
\subseteq [0,1)^s
$$
and the supremum is taken over all such boxes, see~\cite{DrTi,KuNi}.

Typically the bounds on the discrepancy of a 
sequence  are derived from bounds of exponential sums
with elements of this sequence. 
The relation is made explicit in 
 the celebrated {\it Erd\"os-Turan-Koksma
inequality\/}, see~\cite[Theorem~1.21]{DrTi},
which we  present in the following form.

\begin{lemma}
\label{lem:ETK} For any
integer $H > 1$ and any  sequence $\Gamma$ of $N$ points~\eqref{eq:GenSequence}
the discrepancy $\Delta(\Gamma)$
satisfies the following bound:
$$
\Delta(\Gamma) = O\( \frac{1}{H}
+ \frac{1}{N}\sum_{ 0 < |\vec{h}| \le H}  
\prod_{j=1}^s \frac{1}{ |h_j| + 1}
\left| \sum_{n=0}^{N-1} \exp \( 2 \pi i\sum_{j=1}^{s}h_j\gamma_{n,j} \)
\right| \), 
$$
where the sum is taken over all integer vectors
$\vec{h} = (h_1, \ldots, h_s) \in \Z^s$
with $|\vec{h}| = \max_{j = 1, \ldots, s} |h_j| < H$.
\end{lemma}

Finally, we recall that the {\it linear complexity\/} $L$ of 
an $N$-element sequence $s_0, \ldots, s_{N-1}$  in 
a ring $\cR$ is defined as the smallest $L$ such that
$$
s_{u+L}=c_{L-1}s_{u+L-1}+\ldots+c_0s_u, \qquad 
0\le u\le N-L-1,
$$
for some $c_0, \ldots, c_{L-1} \in \cR$, see~\cite{CDR,MOV,TopWin}.

\subsection{Exponential sums}

First, we recall the bound of Heath-Brown~\cite{H-B} on 
exponential sums with $q_p(u)$. Although here we use it 
only with $\nu = 2$ (exactly as it is given in~\cite{H-B})
we formulate it in full generality.

As we have mentioned,   the method of Heath-Brown~\cite{H-B} combined with the
Polya-Vinogradov bound  (when $\nu = 1$)  and the Burgess 
bound (when $\nu \ge 2$), 
see~\cite[Theorems~12.5 and~12.6]{IwKow}, implies the following 
generalisation of~\cite[Theorem~2]{H-B}:

\begin{lemma}\label{lem:HB}
For any fixed integer $\nu \ge 1$, we have
$$
\max_{\gcd(a, p) =1} 
\left|\sum_{u=M+1}^{M+N} 
\ep\(a q_p(u)\) \right| \ll N^{1-1/\nu}p^{(\nu+1)/2\nu ^2+o(1)}, 
$$
as $p\to \infty$, uniformly over $M$ and $N\ge 1$.
\end{lemma}

We now recall the following well-known bound, 
see~\cite[Bound~(8.6)]{IwKow}.

\begin{lemma}
\label{eq:Incompl}
For any integers $K$ and $r$, we have 
$$
\sum_{k=0}^{K-1} \ep(kr) \ll \min\left\{K, \frac{p}{\|r\|}\right\},
$$
where 
$$
\|r\| = \min_{s \in \Z} |r - sp|
$$ 
is the distance between $r$ and the closest multiple of $p$.
\end{lemma}

\subsection{Basic properties of Fermat quotients}

Most of our results are based on the following two
well-known properties of Fermat quotients.

For any integers $k$, $u$ and $v$ with $\gcd(uv,p) = 1$
we have
\begin{equation}
\label{eq:add struct1}
q_p(uv) \equiv q_p(u) + q_p(v) \pmod p
\end{equation}
and
\begin{equation}
\label{eq:add struct2}
q_p(u+kp) \equiv q_p(u) - ku^{-1} \pmod p, 
\end{equation}
see, for example,~\cite[Equations~(2) and~(3)]{ErnMet2}.

\section{Dynamical Properties}

\subsection{Computation of $q_p(u)$}

As we have mentioned, computing each individual 
value of $q_p(u)$ can be done in $O(\log p)$ 
arithmetic operations
on $O(\log p)$-bit integers via repeated squaring computation 
of $u^{p-1}$ modulo $p^2$, we refer to~\cite{vzGG}
for a background on modular arithmetic and 
complexity of various algorithms. 
In particular, one can easily reformulate 
our complexity estimates in terms of 
bit operations.

Thus computing all values of $q_p(u)$, $0 \le u < p$, 
requires $O(p\log p)$  arithmetic operations
on $O(\log p)$-bit integers. Such computation is necessary,
for example, to find all fixed points of the map $u \mapsto q_p(u)$
or for finding the image size. 

Here we show that there is a slightly more efficient
algorithm which is based on~\eqref{eq:add struct1}
and~\eqref{eq:add struct2}.

We assume that we are given a primitive root $g$ modulo $p$.
 This can be done at 
the pre-computation stage and we keep it outside of the algorithm
(in any case, it  can be found in $p^{1/4+o(1)}$ 
arithmetic operations 
on $O(\log p)$-bit integers, see~\cite{Shp1}, which is lower than the
remaining parts of the algorithm).

\begin{algorithm}[Generating $q_p(u)$, $0 \le u \le p-1$] {}\qquad  \newline
  \label{alg:qu all}

\begin{description}

\item {\bf Input:} A prime $p$ and a primitive root $g$ modulo $p$ with $1< g < p$.
\item {\bf Output:} A permuted sequence of the values $q_p(u)$, $0 \le u \le p-1$.
\end{description}

\begin{enumerate}

\item Set $q_p(0) = 0$ and $q_p(1) = 0$.

\item  Compute $q_p(g)$ using the repeated squaring modulo $p^2$.

\item Set $b_1 = g$ and $c_1 = g^{-1} \rem  p$.

\item For $i = 2, \ldots, p-2$ compute

\begin{enumerate}
\item $b_i = g b_{i-1}\rem p$ and $c_i = c_{i-1}g^{-1}\rem p$;
\item $k_i =  (g b_{i-1}-b_i)/p$;
\item  $q_p(b_i)= q_p(g) + q_p(b_{i-1})  +  k_i c_i \rem p$.
\end{enumerate}
\end{enumerate}
\end{algorithm}

\begin{theorem}
\label{thm:qu all}
Algorithm~\ref{alg:qu all}  computes every value 
$q_p(u)$, $0 \le u < p-1$,  in $O\(p\)$
 arithmetic operations
on $O(\log p)$-bit integers.
\end{theorem}

\begin{proof} The complexity estimate is immediate.
The correctness   of the algorithm follows from
the congruences
\begin{eqnarray*}
q_p(b_i)  \equiv  q_p(g b_{i-1} - k_ip)  
& \equiv &  q_p(g b_{i-1})  + k_i (g b_{i-1})^{-1} \\
& \equiv & q_p(g) + q_p(b_{i-1})  +  k_i c_i \pmod p, 
\end{eqnarray*}
which in turn follow from~\eqref{eq:add struct1} 
and~\eqref{eq:add struct2}. 
\end{proof} 

Note that the algorithm of~\cite{ErnMet1,ErnMet2} is very similar,
except that it uses $g=2$ instead of a primitive root. This makes 
each step faster, but if $2$ is not a primitive root 
modulo $p$ requires going trough all conjugacy classes of the 
group generated by $2$ modulo $p$ and thus requires more ``administration''
of data and also more memory.

Unfortunately Algorithm~\ref{alg:qu all} does not help to compute 
$q_p(u)$ for a given value of $u$ unless   all values 
$q_p(v)$, $0 \le v \le p-1$,  are precomputed and stored in a
table, after which $q_p(u)$ can simple be read from 
there. We now describe a trade-off algorithm which 
requires less memory but the   computation of $q_p(u)$
is more expensive than the simple table look-up.
It depends on a  parameter $z\ge 2$,
which can be adjusted to particular algorithmic 
needs.

For a real $V< p$ we use $\cQ_p(V)$ to denote the table of  
the values of $q_p(v)$  with $v \in [0, V]$. 
We see from Theorem~\ref{thm:qu all} that
$\cQ_p(V)$ can be computed in $O\(\min\{p, V \log p\}\)$
arithmetic operations on $O(\log p)$-bit integers.

Furthermore, for an integer $m$, we 
use $\cI_m(V)$ to denote the table of  
the values $v^{-1} \rem m$  with $v \in [1, V]$ and
$\gcd(v,m)=1$.
Since by the Euler theorem $v^{-1} \equiv v^{\varphi(m)-1} \pmod m$, 
where $\varphi(m)$ is the Euler function, we see  that
$\cI_m(V)$ can be computed in $O\(V \log m\)$
arithmetic operations on $O(\log m)$-bit integers
(there are even more efficient modular 
inversion algorithms with a better bound on the 
number of bit operations, see~\cite{vzGG}; however using them 
does not change the overall complexity of our algorithm).

\begin{algorithm}[Computing $q_p(u)$ for a given {$u\in [0, p-1]$}] {}\qquad  \newline
  \label{alg:qu one}

\begin{description}

\item {\bf Input:} A prime $p$, a real $z\ge 2$, the tables $\cQ_p(p/z)$, 
$\cI_p(p/z)$, $\cI_{p^2}(z)$ and 
an integer $u\in \{0, \ldots, p-1\}$.
\item {\bf Output:} The value of $q_p(u)$.
\end{description}

\begin{enumerate}

\item If $u =0$ set $q_p(u) = 0$.

\item Find integers $v$ and $w$ with $u \equiv v/w \pmod p$ and such 
that $1 \le v \le 2p/z$ and $|w| \le z$.

\item   Recall  $r = w^{-1} \rem p^2$ if $w > 0$ 
or  $r = -((-w)^{-1} \rem p^2)$ if $w < 0$ from 
the table $\cI_{p^2}(z)$.

\item  Compute $s$  with $s \equiv v/w \pmod {p^2}$ and 
such that $0\le s < p^2$.

\item Compute $k = (s-u)/p$.

\item  Recall $r = v^{-1} \rem p$ from   the table $\cI_p(p/z)$.

\item  Recall  $q_p(v)$ and $q_p(w)$ from  the   table $\cQ_p(p/z)$.

\item  Compute $q_p(u) = (q_p(v) - q_p(w) +  kr w) \rem p$.
\end{enumerate}
\end{algorithm}

\begin{theorem}
\label{thm:qu one}
For any integer $u$ with $0 \le u < p-1$,
Algorithm~\ref{alg:qu one}  
computes 
$q_p(u)$  in $O\(\log z\)$
arithmetic operations on $O(\log p)$-bit integers.
\end{theorem}

\begin{proof} 
The correctness   of the algorithm follows from
the congruences
\begin{eqnarray*}
q_p(u) & \equiv & q_p(s - kp)  
 \equiv   q_p(s)  + k s^{-1} \\
& \equiv & q_p(v) - q_p(w) +  k v^{-1} w 
\equiv q_p(v) - q_p(w) +  k r w \pmod p
\end{eqnarray*}
which in turn follow from~\eqref{eq:add struct1} 
and~\eqref{eq:add struct2}. 

It remains to estimate  the complexity of finding the $v$ and $w$ 
with $u \equiv v/w \pmod p$. 
We can also assume that $z < p$ since otherwise the
result is trivial.
We start computing 
continued fraction convergents $a_i/b_i$, $\gcd(a_i, b_i)=1$, 
$i =1, 2, \ldots$, to $u/p$, see, for example,~\cite{Ste} for basic 
properties of continued fractions. 
We define $j$ by the condition
$$
b_j \le z < b_{j+1}.
$$
By the well-known property of continued fractions, we have
$$\left| \frac{a_j}{b_j} - \frac{u}{p} \right| \le \frac{1}{b_jb_{j+1}}
\le  \frac{1}{b_jz}.
$$
We now define 
$$
w=|a_jp-b_ju|
$$
and note that (since $z < 0$)
$$0<w = b_jp
\left| \frac{a_j}{b_j} - \frac{u}{p} \right|  
\leq \frac{p}{z}.
$$
Furthermore $u v \equiv w \pmod p$ for either
$v = a_j$ or $v = -a_j$.  Finally, 
since the denominators of the convergents grow
at least exponentially, we see that $j = O(\log b_j) =
O(\log z)$ and thus find $a_j$ and $b_j$ in
$O(\log z)$ steps, each of them requires to compute 
with $O(\log p)$-bit integers. 
\end{proof} 

We see from Theorem~\ref{thm:qu one}
taken with $z = \exp\(\sqrt{\log p}\)$, that
evaluating (in time $p \exp\(-(1 + o(1))\sqrt{\log p}\)$)
and storing $p \exp\(-(1 + o(1))\sqrt{\log p}\)$ values of Fermat 
quotients, we can compute any other value in time 
$(\log p)^{1/2 + o(1)}$.

\subsection{Fixed Points}

Let $F(p)$ denote the number of fixed points of the map $q_p(u)$
that is, 
$$
F(p) = \# \{u \in \{0, \ldots, p-1\}\ : \
q_p(u) = u\}.
$$

We derive a nontrivial estimate on $F(p)$  from
 Lemmas~\ref{lem:ETK} and~\ref{lem:HB}

\begin{theorem}\label{thm:FP}
We have
$$
F(p)  \ll p^{11/12+o(1)}
$$
as $p\to \infty$.
\end{theorem}

\begin{proof} Let us choose some positive integer  
parameter $N \in [1, p-1]$ and for an integer $M$ 
we denote by   $T(p;M,N)$ the number of integers
$u \in [M+1,   M+N]$ with $q_p(u) \in [M+1,   M+N]$.
Considering the discrepancy of the fractions
$q_p(u)/p$, $u = M+1, \ldots, M+N$ and
combining  Lemma~\ref{lem:ETK} (taken with $s=1$)
with Lemma~\ref{lem:HB} (taken with $\nu=2$) , we
immediately conclude 
$$
T(p;M,N) = \frac{N^2}{p} + O\(N^{1/2} p^{3/8+o(1)}\).
$$
Clearly every $u = M+1, \ldots, M+N$ which is a fixed 
point contributes to $T(p;M,N)$. Covering the interval 
$[0,p-1]$ with at most $(p/N + 1)$ intervals of length 
$h$ we obtain 
$$
F(p) \le \(\frac{p}{N} + 1\) \(\frac{N^2}{p} + O\(N^{1/2} p^{3/8+o(1)}\)\).
$$
Choosing $N = \rf{p^{11/12}}$, we conclude 
the proof. 
\end{proof} 

There is little doubt that the bound of Theorem~\ref{thm:FP}
is very imprecise. It is easy to see that in the full range
$0 \le u \le p^2-1$ the relation~\eqref{eq:add struct2}
implies
$$
 \# \{u \in \{0, \ldots, p^2-1\}\ : \
q_p(u) \equiv u \pmod p\} = 2p-1. 
$$
Indeed, it is enough to write $u = v+kp$ with $v,k\in \{0, \ldots, p-1\}$
and notice that 
\begin{itemize}
\item either $v = 0$ and then $k$ can take any values 
\item or $v > 0$ and then the relation~\eqref{eq:add struct2}
identify $k$ uniquely.
\end{itemize}
 Thus one can expect that $F(p) = O(1)$.

 In fact it seems reasonable to expect 
that the map $u\mapsto q_p(u)$ behaves similar to a random 
map. We recall that for a random map on $m$ elements,
the probability of having $k$ fixed points is
$$
\frac{1}{m^m} \binom{m}{k} \times (m-k -1)^{m-k} \to \frac{1}{ek!}
$$
as $m \to \infty$. 

Below we present numerical results giving the numbers $N(k)$
of primes $p\in [50000, 200000]$ for which the map 
$u\mapsto q_p(u)$ has exactly $F(p) = k$ fixed points
(note that we discard the ``artificial'' fixed point $u=0$).
We also give the proportions of such primes $\rho(k) = N(k)/N$
where $N = 12851$ is the total number of primes 
$p\in [50000, 200000]$ and compare them with $\rho_0(k) =(e k!)^{-1}$
for $k =0,\ldots, 6$.
We note that in the above range $N(k)= 0$ for $k \ge 7$.

\begin{center}
\begin{tabular}{|l|l|l|l|l|l|l|l|}
\hline
$k$       & 0  & 1 & 2 & 3  & 4  & 5  &  6  \\
\hline
$\rho_0(k)$ &  0.368 &  0.368 &  0.184  &  0.0613 &  
 0.0153  &  0.00306  &  0.000511\\
\hline
$N(k)$ & 4770  & 4697  & 2327  & 844 &    174  & 36  & 3  \\
$\rho(k)$ & 0.371  & 0.365  & 0.181  & 0.0656 &   
 0.0135   & 0.00280  &  0.000233  \\
\hline
\end{tabular}
\vskip 10pt
{\it Statistics of fixed points}
\end{center}

These numerical results appear to indicate a reasonable agreement between 
the prediction and actual results.

\subsection{Concentration of values} 

For  integers $k$ and $h\ge 1$ we denote 
by $U(p;k,h)$ the number of $u \in \{0, \ldots, p-1\}$
for which $q_p(u) \equiv z \pmod p$ for some $z \in [k+1, k+h]$. 

As in the proof of Theorem~\ref{thm:FP}, a
combination of  Lemma~\ref{lem:HB} (which we take 
with $N =p$ and $\nu =2$) with Lemma~\ref{lem:ETK} 
gives the following asymptotic formula
\begin{equation}
\label{eq:image asymp}
U(p;k,h) = h + O(p^{7/8 + o(1)}) 
\end{equation}
as $p \to \infty$. 
On the other hand, using~\eqref{eq:Rep}, we
trivially  obtain
$$
U(p;k,h) \le hp^{1/2 + o(1)}
$$
that  improves~\eqref{eq:image asymp} for $h \le p^{3/8}$. 

We now obtain a better upper bound, which 
improves~\eqref{eq:image asymp} for $h \le p^{3/4}$.

\begin{theorem}\label{thm:Distr}
For any integers $k$ and $h\ge 1$, we have 
$$
U(p;k,h) \le h^{1/2}p^{1/2 + o(1)}
$$
as $p\to \infty$.
\end{theorem}

\begin{proof} Let $\cU$ be the set of  $u \in \{0, \ldots, p-1\}$, 
which are counted by $U(p;k,h)$. 
Using~\eqref{eq:add struct1} we see that any $w$ of the form $w = uv$
with $uv \in \cU$ satisfies $0 \le w \le p^2 -1$ and 
\begin{equation}
\label{eq:cong w}
q_p(w) \equiv z \pmod p
\end{equation}
for some $z \in [2k+2, 2k+2h]$. 
For a fixed integer $z$, there are $O(p)$ values of 
$w \in \{0, \ldots, p^2-1\}$ satisfying~\eqref{eq:cong w},
which follows immediately from~\eqref{eq:add struct2}
(see also the proof of~\cite[Proposition~2.1]{Fouch}). 
So there are at most $O(hp)$ values of $w$
satisfying~\eqref{eq:cong w} with  some $z \in [2k+2, 2k+2h]$. 
Using the classical estimate 
$$
\tau(w) = w^{o(1)}, \qquad w \to \infty,
$$
on the divisor function $\tau(w)$
(see~\cite[Bound~(1.81)]{IwKow} with $k=2$), 
we deduce that each $w = uv$ can be obtained from no more than $p^{o(1)}$
distinct pairs $(u,v) \in \cU^2$. Therefore 
$\(\# \cU\)^2 \le  hp^{1+o(1)}$, which concludes the proof. 
\end{proof} 

\subsection{Image size} 

Let $M(p)$ be  the image size of the  $q_p(u)$ for $0 \le u \le p-1$,
that is 
$$
M(p) = \# \{q_p(u)\ : \ 0 \le u \le p-1\}.
$$
The bound~\eqref{eq:Rep} immediately implies $M(p) \ge p^{1/2+o(1)}$.
In fact  more precise bounds
$$
\sqrt{p} -1 \le M(p) \le p- \sqrt{(p-1)/2}
$$
can be obtained from~\eqref{eq:add struct1}
and~\eqref{eq:add struct2}, see~\cite[Section~3]{ErnMet2}.

We now obtain a stronger lower bound  on $M(p)$.

\begin{theorem}
\label{thm:Image Size}
We have 
$$
M(p) \ge (1+o(1)) \frac{p}{(\log p)^{2}}, 
$$
as $p\to \infty$.
\end{theorem}

\begin{proof} Let $Q(p,a)$   be the number of primes $\ell \in \{1, \ldots, p-1\}$
with $q_p(\ell) = a$ (note that we have discarded  $u =0$). 
Clearly 
\begin{equation}
\label{eq:1st Moment}
\sum_{a=0}^{p-1} Q(p,a) = \pi(p-1)
\end{equation}
where, as usual, $\pi(x)$ denotes the number of primes $\ell \le x$,
and also 
\begin{equation}
\label{eq:2nd Moment}
\sum_{a=0}^{p-1} Q(p,a)^2 = \# \cR(p), 
\end{equation}
where 
$$
\cR(p) = \{(\ell,r)\ : \ 1 \le \ell,r\le p-1, \ \ell,r~\mathrm{primes}\ q_p(\ell) = q_p(r)\}.
$$
We see from~\eqref{eq:add struct1} that if $(\ell,r) \in \cR(p)$ 
and 
\begin{equation}
\label{eq:wuv}
w\equiv \ell/r\pmod {p^2}
\end{equation}
then 
$$
q_p(w) \equiv q_p(\ell) - q_p(r)   \equiv 0 \pmod p.
$$ 

Since all $w$ with $q_p(w)   \equiv 0 \pmod p$ and $\gcd(w,p)=1$
have 
$$
w^{p-1} \equiv 1 \pmod {p^2},
$$
they are elements of the group $\cG_p$ of the
$p$th power residues modulo $p$. 
Thus we see from~\eqref{eq:wuv} that  
$$
\# \cR(p) \le N(p), 
$$
where $N(p)$ is the number of solutions 
$(\ell, r, w)$  to 
\begin{equation}
\label{eq:wlr}
w\ell \equiv r \pmod {p^2}, \qquad\text {where } \ell,r\le p-1, \
 \ell,r\ \text{primes},\  w\in \cG_p.
\end{equation}

We note that for $w \equiv 1 \pmod {p^2}$ there are exactly
$\pi(p-1)$ pairs $(\ell,r)$  with $\ell = r$ that 
satisfy~\eqref{eq:wlr}.
For any other $w\in \cG_p$ if~\eqref{eq:wlr} is
satisfied for $(\ell_1,r_1)$ and $(\ell_2,r_2)$
then 
$$
\ell_1 r_2  \equiv \ell_2 r_1 \pmod {p^2}
$$
which in turn implies the equation 
\begin{equation}
\label{eq:l12r12}
\ell_1 r_2  = \ell_2 r_1 
\end{equation}
(since $1 \le \ell_1, \ell_2 r_1,r_2\le p-1$). 
Because  $\ell_1, \ell_2 r_1,r_2$ are primes, we 
see from~\eqref{eq:l12r12} that either 
$(\ell_1,\ell_2) = (r_1, r_2)$, which is 
impossible for $w \not \equiv 1 \pmod {p^2}$, 
$(\ell_1,r_1) = (\ell_2, r_2)$, which means that when 
$w\in \cG_p \setminus \{1\}$ is fixed, then~\eqref{eq:wlr}
is satisfied for at most one pair of primes $(\ell,r)$.
Therefore 
\begin{equation}
\label{eq:W bound}
\# \cR(p) \le N(p) \le \pi(p-1) + \#\cG_p-1  = p + O(p/\log p).
\end{equation}

Now, since by the Cauchy inequality we  have 
$$
\(\sum_{a=0}^{p-1} Q(p,a)\)^2 \le M(p) \sum_{a=0}^{p-1} Q(p,a)^2, 
$$
recalling~\eqref{eq:1st Moment} and~\eqref{eq:2nd Moment} 
and using~\eqref{eq:W bound}, we obtain
$$
M(p) \ge (1+o(1)) \pi(p-1)^2 p^{-1}.
$$ 
which concludes  the proof. 
\end{proof}

Clearly the  bound  of Theorem~\ref{thm:Image Size}
is not tight. The image size $M_m$ of a random map on 
an $m$ element 
set is expected to be  
$$
M_m = \(1-\frac{1}{e}\) m = 0.63212\ldots m
$$
see~\cite[Theorem~2]{FlOdl},
and thus it is reasonable to expect that $M(p)/p \approx 1-1/e$.

We now give
the average  value of $M(p)/p$ 
taken over primes $p$ in the intervals 
\begin{equation}
\label{eq:Int Ji}
\cJ_i = [50000i, 50000(i+1)], \qquad i =1,2,3.
\end{equation}
and 
the whole interval 
\begin{equation}
\label{eq:Int J}
\cJ = [50000,  200000].
\end{equation}

\begin{center}
\begin{tabular}{|l|l|l|l|l|}
\hline
Range     & $\cJ_1$ &
$\cJ_2$ & $\cJ_3$  & $\cJ$ \\
\hline
\# of primes & 4459 & 4256 & 4136  &  12851\\
\hline
$M(p)/p$ & 0.63212  & 0.63208  & 0.63212  & 0.63211  \\
\hline
\end{tabular}
\vskip 10pt
{\it Statistics of  image sizes}
\end{center}

\subsection{Distribution of orbit lengths} 

For any map $f$ defined on an $m$ element 
set,
and any initial value $u_0$ from this set, we consider 
the iterations $u_i =f(u_{i-1})$,  $i =1,2,\ldots$.
Then for some $\rho > \mu \ge 0$ we have
$u_\rho = u_\mu$. The smallest value of $\rho$ is
called the {\it orbit length\/} and 
the corresponding (and thus uniquely defined)
value of $\mu$ is called the {\it tail length\/}.

By~\cite[Theorem~3]{FlOdl} the expected values $\rho_m$ and $\mu_m$  of 
the orbit and  tail  length,  taken over all random maps and
initial values $u_0$, satisfy
$$
\frac{\rho_m}{\sqrt{m}} = \sqrt{\pi/2} + o(1)
\mand
\frac{\mu_m}{\sqrt{m}} = \sqrt{\pi/8}+ o(1),
$$
as $m \to \infty$.

Here we present the results of computation of 
the average values of 
the  orbit  and  the tail lengths,  scaled  by $\sqrt{p}$, 
for the sequence~\eqref{eq:FermDyn}
taken over primes $p$ in the intervals $\cJ_1, \cJ_2, \cJ_3$ and
$\cJ$, given by~\eqref{eq:Int Ji} and~\eqref{eq:Int J}, respectively, 
and a randomly chosen initial
value $u_0 \in [1, p-1]$.

\begin{center}
\begin{tabular}{|l|l|l|l|l|}
\hline
Range     & $\cJ_1$ &
$\cJ_2$ & $\cJ_3$  & $\cJ$ \\
\hline
\# of primes & 4459 & 4256 & 4136  &  12851\\
\hline
$\rho/\sqrt{p} $& 1.2423  & 1.2445  & 1.2444  & 1.2437    \\
$\mu/\sqrt{p}$ & 0.62179  & 0.62200  & 0.61806  & 0.62066     \\
\hline
\end{tabular}
\vskip 10pt
{\it Statistics of  orbit  and  the tail lengths, random $u_0$}
\end{center}

Since the values $q_p(2)$ are of special interest, 
we also present similar data where the inutial value is 
alway chosen as  $u_0 = 2$. 

\begin{center}
\begin{tabular}{|l|l|l|l|l|}
\hline
Range     & $\cJ_1$ &
$\cJ_2$ & $\cJ_3$  & $\cJ$ \\
\hline
\# of primes & 4459 & 4256 & 4136  &  12851\\
\hline
$\rho/\sqrt{p} $& 1.2381  & 1.2507  & 1.2401  & 1.2429    \\
$\mu/\sqrt{p}$ & 0.61778  & 0.63004  & .62060  & 0.62275    \\
\hline
\end{tabular}
\vskip 10pt
{\it Statistics of  orbit  and  the tail lengths, $u_0=2$}
\end{center}

The results show quite satisfactory matching with 
the expected values of
$$
\sqrt{\pi/2}  = 1.2533\ldots \mand  \sqrt{\pi/8} = 0.62665 \ldots.
$$

Furthermore, we also give similar average values for $C(p)/p$,
where $C(p)$ is the total number of cyclic points in all possible 
trajectories of the map $u\mapsto q_p(u)$ on the set $\{0, \ldots, p-1\}$, 
taken over primes from the same intervals  $\cJ_1, \cJ_2, \cJ_3$ and $\cJ$.

\begin{center}
\begin{tabular}{|l|l|l|l|l|}
\hline
Range     & $\cJ_1$ &
$\cJ_2$ & $\cJ_3$  & $\cJ$ \\
\hline
\# of primes & 4459 & 4256 & 4136  &  12851\\
\hline
$C(p)/\sqrt{p} $& 1.2413 & 1.2527 &1.23706  &  1.2437  \\
\hline
\end{tabular}
\vskip 10pt
{\it Statistics of  cyclic points}
\end{center}

By~\cite[Theorem~2]{FlOdl} the number $C_m$ of cyclic nodes
of a random map on 
an $m$ element 
set is expected to be  
$$
C_m =\sqrt{\pi/2} m = 1.2533\ldots ,
$$
which again is very close to the observed average values.

\section{Pseudorandomness} 
\label{eq:pseudo}
\subsection{Joint distribution} 

For integers $M$,  $N \ge 1$, $s\ge 1$ and an integer vector 
$\vec{a} = (a_0, \ldots, a_{s-1})$ we consider the exponential sums
$$
S_{s,p}(M,N;\vec{a}) =
\sum_{u=M+1}^{M+N} \ep\(\sum_{j=0}^{s-1} a_j q_p(u+j)\).
$$
Thus the above sums are generalisations of those of Lemma~\ref{lem:HB}
that correspond to the case $s = 1$.  However the method of
Heath-Brown~\cite{H-B} does not seem to apply to the sums
$S_{s,p}(M,N;\vec{a})$ as it requires 
good estimates of mulitiplicative character sums with 
polynomials, which are not currently known (see however~\cite{Chang}
for some potential approaches in the case $s = 2$).

We are now ready to prove an estimate on $S_{s,p}(M,N;\vec{a})$
which together with Lemma~\ref{lem:ETK} implies an 
upper bound on the discrepancy  of 
points~\eqref{eq:Points}. 

\begin{theorem}
\label{thm:Exp Sum}
For any integer $s\ge 1$, we have
$$
\max_{\gcd(a_0, \ldots, a_{s-1}, p) =1} 
\left|S_{s,p}(M,N;\vec{a}) \right| \ll 
s p \log p
$$
uniformly over $M$ and $p^2 > N\ge 1$.
\end{theorem}

\begin{proof} 
Select any $\vec{a} = (a_0, \ldots, a_{s-1}) \in \Z^s$  with
$\gcd(a_0, \ldots, a_{s-1} ,p) = 1$ and take $K=\fl{N/p}$. We get
\begin{eqnarray*}
S_{s,p}(M,N;\vec{a})&=& \sum_{u=M+1}^{M+Kp} \ep\(\sum_{j=0}^{s-1} a_j q_p(u+j)\)+O(p)\\&=&\sum_{u=1}^{Kp} \ep\(\sum_{j=0}^{s-1} a_j q_p(u+M+j)\)+O(p) \\&=&\sum_{v=1}^{p}\sum_{k=0}^{K-1}\ep\(\sum_{j=0}^{s-1} a_j q_p(v+M+j+kp)\)+O(p).
\end{eqnarray*}

Let $\cV$ be the set of $v = 1, \ldots, p$ with 
$v\not \equiv -M - j \pmod p$ for any $j = 0, \ldots, s-1$. 
Therefore, using~\eqref{eq:add struct2}, we obtain:
\begin{equation}
\label{eq:S and W}
S_{s,p}(M,N;\vec{a})  = W +O(p+sK),
\end{equation}
where
\begin{eqnarray*}
W & = & \sum_{v\in \cV}\sum_{k=0}^{K-1}\ep\(\sum_{j=0}^{s-1} (a_j q_p(v+M+j)-a_jk(v+M+j)^{-1})\)
 \\
& = &  \sum_{v\in \cV}\ep\(\sum_{j=0}^{s-1}a_j q_p(v+M+j)\)
\sum_{k=0}^{K-1}\ep\(-k\sum_{j=0}^{s-1}a_j(v+M+j)^{-1})\).
 \end{eqnarray*}
 Taking now the absolute value, we obtain
$$
 \left|W\right|\le \sum_{v\in \cV}
 \left|\sum_{k=0}^{K-1}\ep\(k \sum_{j=0}^{s-1}a_j(v+M+j)^{-1})\)\right|.
$$
 
 Recalling Lemma~\ref{eq:Incompl}, we deduce
$$
 \left|W\right|\le \sum_{v\in \cV} 
 \min\left\{K, \frac{p}{\|F_{\vec{a},s}(v)\|}\right\},
 $$
where 
$$
F_{\vec{a},s}(V) = \sum_{j=0}^{s-1}\frac{a_j}{V+M+j}. 
$$
Examining the poles of $F_{\vec{a},s}(v)$,  we see that 
if $\gcd(a_0, \ldots, a_{s-1}, p) =1$ then it is a nonconstant  
rational function of degree $O(s)$ 
modulo $p$. Thus every residue modulo $p$ occurs $O(s)$ times 
among the values $F_{\vec{a},s}(v)$, $v \in \cV$. Hence
$$
 \left|W\right|\ll s \sum_{u = 0}^{p-1}
 \min\left\{K, \frac{p}{\|u\|}\right\} \ll s p \log p
 $$
 which concludes the proof. 
\end{proof} 
 
%


Using Lemma~\eqref{lem:ETK}, 
we immediately obtain:

\begin{cor}
\label{cor:Exp Sum}
For any fixed $s$, the discrepancy $\Delta_{p,s}(M,N)$ of 
points~\eqref{eq:Points} satisfies 
$$
\Delta_{p,s}(M,N) \ll N^{-1}p(\log p)^{s+1},
$$
uniformly over $M$ and $p^2 > N\ge 1$.
\end{cor}

\subsection{Linear complexity}

Here we estimate the linear complexity for a sufficiently long
sequence of consecutive values of $q_p(u)$.

\begin{theorem}
\label{thm:LCN}
For  $p^2> N \ge 1$ the linear complexity $L_p(N)$ of the sequence 
$q_p(u)$, $u =0, \ldots, N-1$, satisfies
$$
L_p(N) \ge \frac{1}{2} \min\{p-1, N-p-1\}.
$$
\end{theorem}

\begin{proof} Assume that
\begin{equation}
\label{eq:Rec RelN}
\sum_{j=0}^L c_j q_p(u+j) \equiv 0 \pmod p, \qquad 
0\le u\le N-L-1,
\end{equation}
for some integers $c_0, \ldots, c_{L-1}$ and $c_L = -1$. 
Let $R = \min\{p-L, N-L-p\}$.
Then we see from~\eqref{eq:Rec RelN} that for $1 \le u \le R-1$ we have
\begin{equation}
\label{eq:Rel1N}
\sum_{j=0}^L c_j q_p(u+p+j) \equiv 0 \pmod p. 
\end{equation}

Recalling~\eqref{eq:add struct2} and using~\eqref{eq:Rec RelN} again, 
we now see that 
\begin{equation}
\begin{split}
\label{eq:Rel2N}
\sum_{j=0}^L c_j q_p(u+p+j)  \equiv 
\sum_{j=0}^L c_j &\(q_p(u+j) - (u+j)^{-1}\)  \\
 & \equiv  - \sum_{j=0}^L c_j  (u+j)^{-1} \pmod p.
\end{split}
\end{equation}
Comparing~\eqref{eq:Rel1N} and~\eqref{eq:Rel2N} we see that 
$$
 \sum_{j=0}^L c_j  (u+j)^{-1}  \equiv 0 \pmod p, \qquad 1 \le u \le R-1.
$$
We can assume that $L < p$ since otherwise there is nothing to prove. 
Clearing the denominators, we obtain a nontrivial polynomial congruence
$$
 \sum_{j=0}^L c_j \prod_{\substack{h=0\\ h \ne j}}^L (u+h) \equiv 0 \pmod p,
$$
of degree $L$, which has $R-1$ solutions (to see that it is nontrivial 
it is enough to substitute $u=0$ in the polynomial on the left hand side).
Therefore $L \ge R-1$ and the result follows.
\end{proof}

The argument used in the proof of Theorem~\ref{thm:LCN} can also
be used to estimate the linear complexity of arbitrary segments
of the sequence $q_p(u)$, although the resulting bound is slightly weaker.

\begin{theorem}
\label{thm:LCMN}
For $M$ and $p^2> N \ge 1$ the linear complexity $L_p(M;N)$ of the sequence 
$q_p(u)$, $u =M+1, \ldots, M+N$, satisfies
$$
L_p(M;N)\ge \min\left\{ \frac{p-1}{2} , \frac{N-p-1}{3}\right\}.
$$
\end{theorem}

\begin{proof} Assume that
\begin{equation}
\label{eq:Rec RelMN}
\sum_{j=0}^L c_j q_p(u+M+j) \equiv 0 \pmod p, \qquad 
1\le u\le N-L,
\end{equation}
for some integers $c_0, \ldots, c_{L-1}$ and $c_L = -1$. 
Let $R = \min\{p, N-L-p\}$.
Then we see from~\eqref{eq:Rec RelMN} that for $1 \le u \le R$ we have
\begin{equation}
\label{eq:Rel1MN}
\sum_{j=0}^L c_j q_p(u+M+p+j) \equiv 0 \pmod p. 
\end{equation}

Recalling~\eqref{eq:add struct2} and using~\eqref{eq:Rec RelMN} again, 
we now see that for any integer $u$ with 
$u \not \equiv -M - j \pmod p$, $j = 0, \ldots, L$,  we have
\begin{equation}
\begin{split}
\label{eq:Rel2MN}
\sum_{j=0}^L c_j q_p(u+M+p+j)  \equiv 
\sum_{j=0}^L c_j &\(q_p(u+M+j) - (u+M+j)^{-1}\)  \\
 & \equiv  - \sum_{j=0}^L c_j  (u+M+j)^{-1} \pmod p.
\end{split}
\end{equation}
Comparing~\eqref{eq:Rel1MN} and~\eqref{eq:Rel2MN} we see that 
$$
 \sum_{j=0}^L c_j  (u+M+j)^{-1}  \equiv 0 \pmod p, 
$$
for at least  $R-L-1$ values of $u$ with
$$1 \le u \le R\mand u \not \equiv -M - j \pmod p, \ j = 0, \ldots, L.
$$

As before we can assume that $L < p$ since otherwise there is nothing to prove. 
Clearing the denominators, we obtain a nontrivial polynomial congruence
$$
 \sum_{j=0}^L c_j \prod_{\substack{h=0\\ h \ne j}}^L (u+M+h) \equiv 0 \pmod p
$$
of degree $L$, which has at least $R-L-1$ solutions (to see that it is nontrivial 
it is enough to substitute $u=-M$ in the polynomial on the left hand side).
Therefore $L \ge R-L-1$ and the result follows.
\end{proof}

\section{Hash Functions from Fermat Quotients}

\subsection{General Construction} 

In this section we propose a new construction of hash functions based on iterations of Fermat quotients. A similar  idea, however based on 
a very different family of functions,  has 
been previously introduced 
by D.~X.~Charles, E.~Z.~Goren and
K.~E.~Lauter~\cite{CGL}.

Let $n$ and $r$ be two positive integers. Choose $2^r$  random 
$(n+1)$-bit primes  $p_0,\ldots,p_{2^r-1}$. 
We also consider a random initial $n$ bit integer ${u}_0$.

The has function is built from a sequence of iterations of Fermat 
quotients moduli $p_0,\ldots,p_{2^r-1}$.
As in~\cite{CGL}, the input of the hash function is used to decide what 
modulo what prime the next Fermat quotient is computed. More precisely,
given an input bit string $\Sigma$, we perform the following steps:

\begin{itemize}
\item   Pad $\Sigma$ with at most $r-1$ zeros on the left to make
sure that its length $L$ is a multiple of $r$.

\item Split $\Sigma$ into blocks $\sigma_j$, $j =1, \ldots,J$, 
where $J = L/r$, of length $r$ and interpret each block
as an integer $\ell \in [0, 2^r-1]$.

\item Starting at the point $u_0$, apply the Fermat quotient
maps
$q_{p_\ell}$ iteratively by using $n$ least significant bits 
of $u_{j-1}$ to form an $n$-bit integer $w_{j-1}$ and then 
computing
$$
u_{j} =  q_{p_\ell}(w_{j-1}).
$$
  
\item Output the last element in the above sequence, that is, $u_J=q_{p_J}(w_{J-1})$ and outputing its $n$ least significant bits  as the value of the hash function.
\end{itemize}

\subsection{Collision Resistance}

We remark  that  the initial element $u_0$ 
is fixed and in particular, does not depend on the input of the hash function. Furthermore, the collision resistance is based on the difficulty of making the decision which
Fermat quotient to apply at each step when one attempts to back trace 
from a given output to the initial element  $u_0$ and thus
produce two distinct strings $\Sigma_1$ and
$\Sigma_2$ of the same length $L$,  with the same output. 

Note that for strings of different lengths, say of $L$ and $L+1$,  a collision can easily be created. It is enough to take 
$\Sigma_2 = (0, \Sigma_1)$ (that is, $\Sigma_2$ is obtained
from $\Sigma_1$ by augmenting it by $0$). If $L \not \equiv 0 \pmod r$
then they lead to the same output. 
Certainly any  practical implementation has to take care of 
things like this. 

We also note that the results of Section~\ref{eq:pseudo}
suggest that the above hash functions exhibit rather 
chaotic behaviour, which close to the behaviour of
a random function.  It is probably too early to make 
any suggestions about the applicability of Fermat quotients
for hashing but this direction definitely deserves further 
studying, experimentally and theoretically.

\section{Comments}

Unfortunately we are not able to give any estimates 
on the discrepancy or linear complexity of the 
orbits~\eqref{eq:FermDyn}, which is a very interesting
but possibly hard, question.

Obtaining analogues of Theorems~\ref{thm:Exp Sum}, 
\ref{thm:LCN} and~\ref{thm:LCMN}, 
which are nontrivial for $N < p$ is another interesting question. 

The method of proof of Theorems~\ref{thm:LCN} 
and~\ref{thm:LCMN} does not apply 
to the {\it nonlinear complexity\/}. We recall 
the nonlinear complexity of degree $d$ of 
an $N$-element sequence $s_0, \ldots, s_{N-1}$  of elements in
a ring $\cR$
is the smallest  $L$ o such that  
$$
s_{u+L}=\psi(s_{u+L-1},\ldots,s_u), \qquad 
0\le u\le N-L-1,
$$
where $\psi\in\cR[Y_1,\ldots,Y_L]$ is a polynomial of 
total degree at most $d$. Estimating the nonlinear complexity
of Fermat quotients is of ultimate interest. 

Finally, we remark that one can also study 
the sums
$$
T_{p}(M,N;\chi) =
\sum_{u=M+1}^{M+N} \chi\(q_p(u)\)
$$
with a nonprincipal multiplicative character $\chi$ modulo $p$. 
Arguing as in the proof of Theorem~\ref{thm:Exp Sum} 
we get 
$$
|T_{p}(M,N;\chi)| \ll
\sum_{v= M+1}^{M+p-1}
 \left|\sum_{k=0}^{K-1}\chi\(q_p(v+M) -k (v+M)^{-1})\)\right| + p, 
$$
where $K = \fl{N/p}$. 
One can now apply the  Burgess bound, 
see~\cite[Theorems~12.6]{IwKow},
and get a nontrivial estimate on $T_{p}(M,N;\chi)$, 
starting with $N \ge p^{5/4 + \varepsilon}$ for 
any fixed $\varepsilon > 0$, see~\cite{Shp2}.  
However it is natural to expect 
that one can take advantage of additional
averaging over $v$ and get a nontrivial bound for smaller
values of $N$. Furthermore, using~\eqref{eq:add struct1}
it is possible to estimate bilinear character sums
$$
W_p(\cA, \cB, U,V;\chi) = \sum_{0 \le u \le U} \sum_{0 \le v \le V}
\alpha_u \beta_v \chi\(q_p(uv)\)
$$
with arbitrary complex weights $\cA = \(\alpha_u\)$ 
and $\cB = \(\beta_v\)$, and then using 
the Vaughan identity, see~\cite[Section~13.4]{IwKow},
estimate the character sums with Fermat quotients at  primes
arguments, see~\cite{Shp2} for details. 

Furthermore,  we remark that studying  the map 
$x \mapsto (x^{p-1} - 1)/p$ in the field of $p$-adic numbers, 
is also of great interest,  see~\cite{SmWo}  where a similar 
question is considered for the maps 
given by~\eqref{eq:L fun}. The other way around, it is also quite 
natural to study the map~\eqref{eq:L fun} modulo $p$.  

Finally, analogues of Fermat quotients modulo a composite 
number is certainly an exciting object of study with its
own twists, see~\cite{Agoh,ADS,BLS,Dilch}.

\end{document}